\documentclass{amsart}
\usepackage{amsmath, amssymb}
\usepackage{comment}
\usepackage[all]{xypic}
\newcommand{\eeq}{\end{equation}}
\DeclareMathOperator{\PGL}{PGL}

\newcommand{\bL}{ \bar L}
\newcommand{\lra}{\longrightarrow}

\newcommand{\bA} {\mathbb{A}}
\newcommand{\bP}{\mathbb{P}}
\newcommand{\bZ} {\mathbb{Z}}

\newcommand{\bQ}{\mathbb{Q}}

\newcommand{\cF}{\mathcal{F}}
\newcommand{\CL}{\mathcal{L}}
\newcommand{\cO}{\mathcal{O}}
\newcommand{\cU}{\mathcal{U}}

\newcommand{\vb}{\vec{\beta}}

\DeclareMathOperator{\Prep}{Prep}
\newcommand{\ta}{\tilde \alpha}
\newcommand{\tG}{\tilde G}
\newcommand{\tb}{\tilde \beta}
\newcommand{\tg}{\tilde \gamma}

\DeclareMathOperator{\Spec}{Spec}
\DeclareMathOperator{\Aut}{Aut}
\DeclareMathOperator{\Per}{Per}
\DeclareMathOperator{\Res}{Res}

\newcommand{\Kbar}{ \overline{K}}

\newcommand{\fm}{\mathfrak m}
\newcommand{\fo}{\mathfrak o}
\newcommand{\cX}{\mathcal X}

\newcommand{\cY}{\mathcal Y}

\newtheorem{thm}{Theorem}[section]

\newtheorem{prop}[thm]{Proposition}
\newtheorem{lem}[thm]{Lemma}
\newtheorem{cor}[thm]{Corollary}
\newtheorem{quest}[thm]{Question}

\theoremstyle{definition}
\newtheorem{defn}[thm]{Definition}

\newtheorem{example}[thm]{Example}

\title[Preperiodic points, finiteness, and structures]{Preperiodic
  points, finiteness, and structures of semigroups of algebraic
  automorphisms}

\author{Thomas J. Tucker}
\address{Thomas J. Tucker\\Department of Mathematics\\ University of Rochester\\
  Rochester, NY, 14620, USA}
\email{tjtucker@gmail.com}

\author{Jason P. Bell}
\address{Jason P . Bell \\ University of Waterloo \\ 200 University
  Ave. W., \\ Waterloo, ON, N2L 3G1, CANADA}
\email{ jpbell@uwaterloo.ca}

\begin{document}
\maketitle

\begin{abstract}

  In this paper, we explore a variety of finiteness questions for
  preperiodic points of morphisms.  We begin by treating a group
  action analog of the Burnside problem for torsion groups using the
  $p$-adic arc method.  We then prove some results connecting
  commonality of preperiodic points for elements of an automorphism
  group with structural properties of the group; these results are
  related to well-known results of Tits and Borel.  We finish by
  proving some Northcott-type results for finite morphisms.

\end{abstract}

\section{Introduction}\label{intro}

Beaumont \cite{Beaumont} recently proved the following theorem.

\begin{thm}
Let $f,g: \bP_K^n \lra \bP_K^n$ where $\deg f, \deg g > 1$ and $K$ is
any field.  If $\Prep(f) \not= \Prep(g)$, then the semigroup generated
by $f$ and $g$ under composition is the free semigroup on two generators.
\end{thm}

This built on earlier work of \cite{BHPT}, who showed that under the
same hypotheses, there is an $m$ such that the semigroup generated by
$f^m$ and $g^m$ under composition is the free semigroup on two
generators.  The idea of Beaumont's proof, which follows that of
\cite{BHPT}, is to use the fact that $f$ and $g$ induce contraction
maps on a suitable space of height functions and then show that any two
contraction maps satisfying a natural restriction on their multipliers
will generate a free semigroup on two generators unless they have a
common fixed point (which happens for the contraction maps induced by
$f$ and $g$ exactly when $\Prep(f) = \Prep(g))$.  We note that
Beaumont's theorem actually holds for any two morphisms of a projective
variety $V$ polarized by the same line bundle (that is, for any maps $f,g:
V \lra V$ such that there is an ample line bundle $\CL$ on $V$ with the
property that $f^* \CL \cong \CL^{d_1}$ and $g^*\CL \cong \CL^{d_2}$ for some
$d_1, d_2 > 1$).

The conclusion of Beaumont's theorem does not hold as stated if one
of the maps is an automorphism.  Consider the semigroup $S$
generated by $f(x) = 2x$ and $g(x) = x^2$.  We see that $\Prep(g)$ is
much larger than $\Prep(f)$, but we have
$f \circ f \circ g = g \circ f$, so $S$ is not free on two generators.
On the other hand, $S$ does contain $2x^2$ and $x^2$, which do generate a
free semigroup on two generators (this follows easily from the
uniqueness of binary expansions of natural numbers).

Additional complications occur for morphisms having periodic points
that are not isolated, as we shall see in Examples \ref{nil-ex} and
\ref{solv-ex}.  To explain this, we need a little terminology.

\begin{defn}
For a given integer $n$ and an automorphism $g: X \lra X$ of a variety
$X$, we define

\[  \Per_n(g) = \{ x \in X(\Kbar) \; | \; \text{$g^n(x) = x$}
    \}. \]

  Note that $\Per_n(g)$ is a Zariski closed subset of $X$.  We define
  $\Per(g)$ to be the union $\cup_{n=1}^\infty \Per_n(g)$.  

  We let $\Per_n^*(g)$ be the points $x$ that are in
  0-dimensional components of $\Per_n(g)$, where $n$ is the smallest
  positive integer for $g^n(x) = x$. We define $\Per^*(g)$ to be the
  union $\cup_{n=1}^\infty \Per^*_n(g)$. 

We define $\Per^{**}(g)$ to be the set of all $x \in X(\Kbar)$ such
that $x \in \Per(g)$ and $x$ is not in a positive-dimensional
component of $\Per(g^n)$ for any $n$.  Note that this may be smaller
than $\Per^*(g)$ since a point could be in a zero-dimensional
component of $\Per_n(g)$ but in a positive dimensional component of
$\Per_{mn}(g)$ for some $m > 1$ (for example, $g$ could itself be torsion).  

\end{defn}

Many questions related to the problems in this paper, such as the Tits
alternative \cite{Tits} and its variants (see \cite{Tits, BHPT,
  Tits-Var, Tits-Var-Z, Hu}) are stated in terms of finitely generated
groups or finitely generated semigroups, so we now pose a general
question in that language.

\begin{quest}\label{TitsQ1}
  Let $G$ be a finitely generated group of automorphisms of a
  quasiprojective variety $X$.  It is true that one of the following
  must hold:
\begin{itemize}
\item[(i)] we have $\Per^{**}(g_1) \subseteq \Per(g_2)$ for any $g_1,
  g_2 \in G$ or
\item[(ii)] $G$ contains a subsemigroup that is isomorphic to the free
  semigroup on two generators?
\end{itemize}
\end{quest}

We are able to answer Question \ref{TitsQ1} in the case of projective
varieties.

\begin{thm}\label{auto-proj}
  Let $G$ be a finitely generated group of automorphisms of a
  projective variety $X$.  Then at least one of the following must hold:
\begin{itemize}
\item[(i)] we have $\Per^{**}(g_1) \subseteq \Per(g_2)$ for any $g_1,
  g_2 \in G$; or
\item[(ii)] $G$ contains a subsemigroup that is isomorphic to the free
  semigroup on two generators.
\end{itemize}
\end{thm}

We prove Theorem \ref{auto-proj} by combining Tits type results for
automorphism groups of projective varieties \cite{Tits-Var,
    Tits-Var-Z, Hu} and a result of Rosenblatt's on solvable groups
  \cite{Rosenblatt} with the following theorem proved in Section \ref{nil-sec}.

  \begin{thm}\label{nil-same}
  Let $G$ be a finitely generated group of automorphisms of a variety
  $X$.  If $G$ is nilpotent, then for any
  $g_1, g_2 \in G$, we have $\Per^{*}(g_1) \subseteq \Per(g_2)$.  In
  particular, if $\Per_n(g_1)$  and $\Per_n(g_2)$ are finite for all
  $n$, then $\Per(g_1) = \Per(g_2)$.  
 \end{thm}

 Theorem \ref{nil-same} is a statement relating commonality of
 periodic points among elements of a nilpotent group; it seems natural
 to ask for a theorem relating commonality of periodic points among
 elements of a solvable group.  In the case of linear algebraic
 groups, the Borel fixed point theorem \cite{Borel1, Borel2} says that
 if $G$ is a connected, solvable, linear algebraic group acting
 regularly on a nonempty, complete algebraic variety $V$ over an
 algebraically closed field $K$, then there is an algebraic point in
 $V(K)$ that is fixed by every element of $G$.  Using techniques
 similar to those of the proof of Theorem \ref{nil-same} along with
 some $p$-adic techniques from \cite{BGT1, BGT2} (that are available
 only in characteristic 0), we prove the following.  Recall that a
 group is said to be virtually solvable (resp. nilpotent) if it
 contains a subgroup of finite index that is solvable
 (resp. nilpotent).  
 
 \begin{thm}\label{solvable}
  Let $G$ be a finitely generated virtually solvable group of automorphisms of a
  variety $X$ defined over a field of characteristic 0.  Suppose that
  for all nontorsion $g \in G$, the set $\Per^{*}(g)$ is nonempty.
  Then there is an $x \in X$ such that $x \in \Per(g)$ for all $g \in
  G$.  Furthermore there is a subgroup $H$ of finite index in $G$ such
  that $h(x) = x$ for all $h \in H$.  
 \end{thm}
  
An important ingredient of the proof of Theorem \ref{solvable} is the
following theorem, which may be thought of as a Burnside-type problem
for group actions.  Recall that the Burnside problem for groups asks
if a finitely generated torsion group is finite; there are
counterexamples in general, but with additional requirements (say on
the order of the torsion and the number of generators), it is often
true.  Likewise, given a group action $G$ on a set $T$, one can ask when
it is true that $Gt$ must be finite for any $t \in T$ that is periodic
under every point of $G$.  When $T$ is taken to be $G$ itself under
left multiplication and $t = \{ e \}$, this is the same as the Burnside
problem for groups.  We are able to prove the following Burnside-type
theorem for groups acting on varieties using $p$-adic techniques in
Section \ref{burn-sec}.  

\begin{thm}\label{arith-burn}
  Let $X$ be a quasiprojective variety over a finitely generated field
  $K$ of characteristic 0 and let $G$ be a group of automorphisms of
  $X$ defined over $K$.  Let $W$ be an irreducible subvariety of $X$.
  If $W$ is periodic under $g$ for every element $g \in G$, then the
  orbit of $W$ under $G$ is finite.
\end{thm}

An outline of the paper is as follows. We begin by proving Theorem
\ref{arith-burn} in Section \ref{burn-sec}; the technique is quite
similar to that of \cite{BGT2}.  Then, in Section \ref{nil-sec}, we
prove Theorems \ref{nil-same} and \ref{solvable}.  The proof of
Theorem \ref{nil-same} uses the fact that the center of a
finitely generated infinite nilpotent group is infinite, along with some simple
relationships between periodic points of commuting automorphisms that
hold in any characteristic, while Theorem \ref{solvable} relies on a
refined version of Theorem \ref{arith-burn} that holds for points
(Proposition \ref{burn1}) in characteristic 0.  Finally, in Section
\ref{north}, we propose a more general version of Question
\ref{TitsQ1} and prove two Northcott-type results that might be useful
in working towards its proof.

Throughout this paper, our varieties are not assumed to be
irreducible.  When a variety is irreducible we refer to it as an {\em
  irreducible} variety.

\noindent {\bf Acknowledgements.}  We would like to thank Gabriel Vigny and
Keping Huang for helpful conversations.

\section{Burnside orbit theorem}\label{burn-sec}

We begin with a lemma adapted from \cite{BGT2}.  As there, the main tool
is the $p$-adic arc lemma.  While the emphasis there was on bounding
the periods of periodic points, here we focus instead on bounding
the index of the stabilizer subgroup of a subvariety that is periodic under every
element of some group of automorphisms.  It is possible here to make
the bounds effective in the case where $W$ is a point.

\begin{lem}\label{basic}
  Let $\fo_v$ be the integral closure of $\bZ_p$ in some finite extension
  $K_v$ of $\bQ_p$ for some $p$, let $\fm_v$ be the maximal ideal in
  $\fo_v$, and let $k_v$ be the residue field $\fo_v/\fm_v$.  Let
  $\cX$ be a smooth $\bZ_p$-scheme whose generic fiber is a variety
  $X$ over $K_v$.  Let $\gamma \in \cX(k_v)$ and let
  $\cU_\gamma = r^{-1}(\gamma)$, where
  $r: \cX(\fo_v) \lra \cX(k_v)$ is the usual reduction map.   Let
  $G_\gamma$ be a group of automorphisms $\sigma$ of $\cX$ such
  that $\sigma(\cU_\gamma) = \cU_\gamma$.  Then $G_\gamma$ has a
  subgroup $H_\gamma$ of finite index such that $h(W) = W$ for any
  irreducible subvariety $W$ of the generic fiber of $\cX$ such that
  the closure of $W$ in $\cX$ contains a point in $\cU_\gamma$.     
\end{lem}
\begin{proof}
  The proof is very similar to the proof of \cite[Theorem
  3.1]{BGT2}.  Let $\fm_v$ denote the maximal ideal in $\fo_v$, let
  $\pi$ denote a generator for $\fm_v$,  and let
  $k_v$ denote the residue field $\fo_v/\fm_v$.  
  Since $\cX$ is a smooth $\fo_v$-scheme, for any $\alpha \in
  \cX(k_v)$, its completed local ring  ${\hat \cO_{\alpha}}$ is
  isomorphic to $\fo_v[[z_1,\dots, z_n]]$, where $n$ is the dimension
  of $X$.  

  Arguing exactly as in the proof of \cite[Proposition~2.2]{BGT1} we then
obtain that there is a $v$-adic analytic
isomorphism $\iota: \cU_\gamma \lra \fo_v^n$, such that for any $\sigma
\in G_\gamma$, there are power series $F_1, \dots, F_n \in \fo_v[[z_1,
\dots, z_n]]$ with the properties that
\begin{enumerate}
\item[(i)] each $F_i$ converges on $\fo_v^n$;
\item[(ii)] for all $(\beta_1, \dots, \beta_n) \in \fo_v^n$, we have
\begin{equation}\label{power eq}
\iota(\sigma( \iota^{-1} (\beta_1, \dots, \beta_n))) =   (F_1(\beta_1,\dots, \beta_n),\dots, F_n(\beta_1,
\dots, \beta_n)); \text { and }
\end{equation}
\item[(iii)] each $F_i$ is congruent to a linear polynomial mod $v$ (in
  other words, all the coefficients of terms of degree greater than one
  are in the maximal ideal $\fm_v$ of $\fo_v$). Moreover, for each $i$, we have
$$F_i(z_1,\dots,z_n)=\frac{1}{\pi}\cdot H_i(\pi  z_1,\dots, \pi z_n),$$
for some $H_i\in \fo_v[[z_1,\dots, z_n]]$.
\end{enumerate}

Writing $\vb: = (\beta_1, \dots, \beta_n)$ and $\cF_\sigma:=\iota
\sigma \iota^{-1}$, we thus have
\begin{equation}\label{p}
\cF_\sigma (\vb) \equiv C_\sigma+L_\sigma (\vb) \pmod{\pi}
\end{equation}
for a $C_\sigma\in\fo_v^n$ and an $n \times n$ matrix $L_\sigma$ with
coefficients in $k_v$.  Since $\sigma$ is an \'etale map of
$\fo_v$-schemes, $\bL_\sigma$ must be invertible.  Let $H_\gamma$ be
the subgroup of $G$ consisting of all $\sigma \in G$ such that
$L_\sigma \equiv I_n \pmod {\pi^{e+1}}$ where $I_n$ is the
$n \times n$ identity matrix with coefficients in $k_v$ and $e$ is the
ramification index of $K_v$ over $\bZ_p$.  Then $H_\gamma$ has finite
index in $G$ since $n$ and $\fo_v / \fm_v^{e+1}$ are finite.  We will
show that $h(W) = W$ for any $h \in H_\gamma$.  Pick
$z \in \cU_\gamma \cap W$ and $\sigma \in H_\gamma$ and let
$\vb = \iota(z)$.  By \cite{Poonen}, there are $p$-adic power series
$\theta_1, \dots, \theta_n: \bZ_p \lra U_\gamma$ convergent on $\bZ_p$
such that $\cF^n(\vb) = (\theta_1(m), \dots, \theta_n(m))$ for all
$m \in \bZ_p$.

Now, let $G$ be any element of the vanishing ideal of $W$.  Then
$G \circ \iota^{-1}$ is a $v$-adic analytic power series
convergent on $\fo_v^g$, so
\[ \tG = (G \circ \iota^{-1}) (\theta_1(z), \dots, \theta_n(z))\] is a
$p$-adic analytic power series convergent on $\bZ_p$.

If $z \in W$ and $W$ is periodic, then $\tG(m)$ must vanish on all $m$
divisible the period of $W$, so $\tG$ is a power series convergent on
$\bZ_p$ that has infinitely many zeros.  This means that $\tG$
vanishes uniformly on $\bZ_p$, so $G$ must vanish on $\sigma^m(z)$ for
all natural numbers $m$.  Since this holds for every $G$ in the
vanishing ideal of $W$, we must have $\sigma^m(z) \in W$ for all $m$
so in particular, we have $\sigma(z) \in W$.

Now, $W \cap \cU_\gamma$ is dense by \cite[Proposition 2.3]{BGT2}, so
it follows that $\sigma(x) \in W$ for a Zariski dense subset of $x \in W$,
so $\sigma(W) = W$, as desired.

\end{proof}

In the case of points, we obtain a little bit more information.  

\begin{prop}\label{burn1}
Let $\fo_v$ be the integral closure of $\bZ_p$ in some finite extension
$K_v$ of $\bQ_p$ for some $p$.  Let $\cX$ be a $\bZ_p$-scheme
whose generic fiber is a variety $X$ over $K_v$, let $G$ be a group
of automorphisms of $\cX$, and let $T$ be the set
of all points in $\cX(\fo_v)$ that are periodic under every element of
$G$.  Let $M$ be the set of all elements of $G$ such that $ht = t$ for
all $t \in T$.  

Then we have the following:
\begin{itemize}
\item [(i)] $M$ is normal in $G$;
\item[(ii)] $[G:M]$ is finite; and
  \item [(iii)] for any $t \in T$, the orbit $Gt$ is finite. 
\end{itemize}
\end{prop}
\begin{proof}
  We see that (iii) follows immediately from (ii) since for any
  $t \in T$, the stabilizer $G_t$ of $t$ contains $M$.  Note that we
  may assume that $X$ is geometrically irreducible in (i) since the
  stabilizer of each component of $X$ has finite index in $G$.
  Likewise, (i) follows easily from (iii) since  if $Gt$ is finite
  for all $t \in T$, then $gt \in T$ for any $g
  \in G$ and $t \in T$, so $ghg^{-1}(t) = g g^{-1} t = t$ for
  any $h \in M$, which means that $g h g^{-1} \in M$ for any $g \in G$
  and $h \in M$.   Hence, it only remains to prove (ii).  
  
We use induction on the dimension of $X$ and Lemma \ref{basic} to prove
(ii).  If $X$
has dimension 0, this is trivial.  Otherwise, we may write $\cX$ as
$\cX' \cup \cY$ where $\cX'$ is smooth over $\bZ_p$ and the generic
fiber of $\cY$ has dimension less than $\dim X$.  An automorphism of a
scheme preserves its smooth locus so $\sigma(\cX') = \cX'$ and
$\sigma(\cY) = \cY$.  We let $M_Y$ be the subgroup of of $G$ that
fixes all elements of $T$ in $\cY$ and $M'$ be the subgroup of $G$
that fixes all elements of $T$ in $\cX'$.  Then $[G:M_Y]$ is finite,
by the inductive hypothesis, so it will suffice to show that $[G:M']$
is finite
since $[G:M_Y \cap M'] \leq [G:M'] \cdot [G:M_Y]$.  

Now, since there are finitely many $\gamma \in \cX'(k_v)$, there are
finitely many residue classes $\cU_\gamma$ of $\fo_v$-points as in
Lemma \ref{basic} containing elements of $T$.  Since $G$ permutes the
residue classes $\cU_\gamma$, there is a subgroup $H$ of finite index
in $G$ such that $h(\cU_\gamma) = \cU_\gamma$ for all $h \in H$ and
all residue classes $\cU_\gamma$. By Lemma \ref{basic}, for each
$\gamma \in \cX'(k_v)$, there is a subgroup $H_\gamma$ of finite index
in $H$ such that $ht =t$ for all $t \in \cU_\gamma$ that are periodic
under all elements of $H$.  Let
$M'' = \cap_{\gamma \in \cX'(k_v)} H_\gamma$.  Now if $t \in T$, then
$t$ is periodic under every element of $H$ and furthermore
$t \in \cU_\gamma$ for some $\gamma \in \cX'(K_v)$; hence, $M''$ fixes
every element of $T$.  Since $M''$ is an intersection of subgroups of
finite index in $G$, it has finite index in $G$.  Clearly, $M'$
contains $M''$ so $[G:M']$ is also finite, as desired.
\end{proof}

The proof of Theorem \ref{arith-burn} is now essentially identical to
the proof of \cite[Theorem 1.2]{BGT2}.

\begin{proof}[Proof of Theorem \ref{arith-burn}]
  As in the proof of Proposition \ref{burn1}, it suffices to show this
  when $X$ is geometrically irreducible.  Let
  $\sigma_1,\dots, \sigma_m$ be a finite set of generators for $H$,
  and let $K$ be a finitely generated field such that $X$,
  $\sigma_1,\dots, \sigma_m$ are all defined over $K$.  We may further
  assume that some point $\alpha \in W$ is defined over $K$ (if $W$ is
  zero-dimensional, this will just be $\alpha$ itself).  Let $R$ be a
  finitely generated $\bZ$-algebra containing all the coefficients of
  all the polynomials defining $X$ in some projective space, along
  with all the coefficients of all the polynomials defining all the
  $\sigma_i$ locally, as in the proof of \cite[Theorem 4.1]{BGT1}.  By
  \cite[Proposition 4.3]{BGT1}, since a finite intersection of dense
  open subsets is dense, we see that there is a dense open subset $U$
  of $\Spec R$ such that:
\begin{enumerate}
\item there is a scheme $\cX_{U}$ that is quasiprojective over $U$, and whose generic fiber equals $X$;
\item each fiber of $\cX_{U}$ is geometrically irreducible;
\item each $\sigma_i$ extends to an automorphism ${\sigma_i}_{U}$ of $\cX_{U}$; and
\item $\alpha$ extends to a smooth section $U \lra \cX_U$.
\end{enumerate}
Now, arguing as in \cite[Proposition 4.4]{BGT1}, and using \cite[Lemma
3.1]{Bell}, we see that there is an embedding of $R$ into $\bZ_p$ (for
some prime $p\ge 5$), and a
$\bZ_p$-scheme $\cX_{\bZ_p}$ such that
\begin{enumerate}
\item $\cX_{\bZ_p}$ is quasiprojective over $\bZ_p$, and its
  generic fiber is isomorphic to $X \times_K \bQ_p$;
\item both the generic and the special fiber of $\cX_{\bZ_p}$ are geometrically irreducible;
\item each $\sigma_i$ extends to an automorphism $(\sigma_i)_{\bZ_p}$ of $\cX_{\bZ_p}$; and
\item $\alpha$ extends to a smooth section $ \Spec \bZ_p \lra \cX_{\bZ_p}$.
\end{enumerate}

Let $r_v(\alpha) = \gamma$ and let $H$ be the subgroup of $G$
consisting of all $h \in G$ such that $h(\cU_\gamma) = \cU_\gamma$
where $\cU_\gamma = r_v^{-1}(\gamma)$ as before.  Then by Lemma
\ref{basic}, there is a finite index subgroup $H_\gamma$ of $H$ such
that $h(W) = W$ for all $h \in H$.  Since $H_\gamma$ has finite index
in $G$, we see that $GW$ must be finite.

\end{proof}

We pose the following general question.

\begin{quest}\label{BurnsQ}
Let $S$ be a finitely generated group of finite morphisms from a
variety $X$ to itself.  Suppose that $x \in S$ is preperiodic under
every element of $S$.  Is it true that the orbit $Sx$ must be finite?
 \end{quest} 

We are able to prove that Question \ref{BurnsQ} has a positive answer
for self-maps of projective space.  
 
\begin{thm}\label{polar-burn}
  Let $S$ be a finitely generated semigroup of morphisms
  $f: \bP_K^n \lra \bP_K^n$ for some $n$ defined over a field $K$.
  Then for any $x \in \bP^n$ such that $x \in \Per(g)$ for all
  $g \in S$, the orbit $Sx$ is finite.
\end{thm}
\begin{proof}

  We first treat the case where $S$ is composed entirely of
  automorphisms (but is not necessarily a group itself).  We will
  begin by showing that there is an $M$ such that $s^Mx = x$ for all
  $s \in S$.  If $K$ has characteristic $p$, then for any
  $(n+1) \times (n+1)$ matrix $N$, we have that $N^{p^{n+1}}$ is
  diagonalizable (this can be seen by putting it in Jordan canonical
  form).  Since fixed points of elements of $\PGL_{n+1}(K)$ correspond
  to eigenvectors of $(n+1) \times (n+1)$ matrices, it follows that
  for any $g \in \PGL_{n+1}(K)$ and any $y \in \Per(g)$, we have
  $g^{p^{n+1}} y = y$ when $K$ has characteristic $p$; thus, we have
  $s^{p^{n+1}} x = x$ for all $s \in S$.  When $K$ has characteristic
  $0$, we begin by arguing as in the proof of Theorem
  \ref{arith-burn}. Let $R$ be a finitely generated $\bZ$-algebra
  containing all
  the coefficients of all the polynomials defining some set of
  generators for $S$ locally.  Then, again by \cite[Proposition
  4.3]{BGT1} there is a $p \geq 5$ and a model $\cX$ for $X$ over
  $\bZ_p$ such that $x$ extends to a point in $\bP_{\bZ_p}^n(\bZ_p)$
  and each element of $S$ extends to an automorphism of
  $\bP_{\bZ_p}^n$.  Then by \cite[Theorem 1.1]{BGT2}, there is an $M$
  (depending only on $K$, $n$, and a place of good reduction for the
  generators of $S$) such that $s^M(x) = x$ for all for all $s \in S$.

  Let $T$ be the Zariski closure
  of $S$ in $\Aut(\bP^n)$.  Then $T$ is clearly a semigroup, so for
  any $t \in T$, the sets $t^n T$ form a chain
  \[ T \supseteq tT \supseteq \cdots \supseteq t^n T \supseteq
    \cdots \] of Zariski closed subsets of $T$.  Any such chain
  eventually stabilizes (by finiteness of dimension), so there are
  $i > j$ such that $t^i T = t^j T$, which means that $t^{i-j}$ is a
  unit in $T$ and thus that $a$ is a unit in $T$.  Hence, $T$ is a
  group containing the group $G$ generated by $S$.  Since the
  condition $g^M(x) = x$ defines a closed subset of elements
  $g \in\Aut(\bP^n)$, we must have $g^M(x) = x$ for all $g \in A$ and
  thus for all $g \in G$. In characteristic 0, it then follows from
  Theorem \ref{arith-burn} that $Gx$ is finite.

  When we are in characteristic $p$, we need slightly more complicated
  argument to finish our proof.  Let $H$ be the subgroup generated by
  all $M$-th powers of elements of $G$.  Note that $H$ is normal
  (since a conjugate of an $M$-th power is itself an $M$-th power), so
  every element fixes not only $x$ but also every element in its orbit
  $Gx$. Let $V$ be a component of the Zariski closure of $Gx$.  Then
  $H$ fixes every element of $V$ as well, and the subgroup $G_V$ of
  $G$ that sends $V$ to itself must have finite index in $G$, since
  $G$ has finitely many components.  The group $G_V/H$ acts on $V$ and
  is torsion with bounded exponent.  We let $A$ denote $G_V/H$ and let
  $m$ denote the size of some generating set for $A$ (note that $G_V$
  is finitely generated by the Schreier index formula since it has
  finite index in $G$ and $G$ is finitely generated).  We will show
  that $A$ is finite, which will imply that $H$ has finite index in
  $G$ and complete our proof.  First, since $A$ is a subgroup of
  $\Aut(V)$, it is residually finite (that is, for any element of
  $g \not= 1$ in $A$, there is a map from $A$ to a finite group that
  does not send $g$ to the identity) by \cite[Theorem 1.1]{Bass2}, so
  there is an embedding
  \[ \theta: A \lra \prod_{i=1}^\infty B_i \]
  where each $B_i$ is a finite group.  For any $j$ let $\rho_j: \prod_{i=1}^\infty B_i
\lra  \prod_{i=1}^j B_i$ be the projection map onto the first $j$
coordinates.  Then we have an ascending chain of groups
\begin{equation}\label{chain}
\rho_1\circ\theta (A) \subseteq \rho_2\circ \theta(A) \subseteq
\cdots. 
\end{equation}

Now each group $\rho_j \circ \theta(A)$ is finite with
exponent $M$ and is generated by $n$ elements. Note that $M$ is a
power of $p$, so we may apply Zelmanov's theorem \cite{Z1,Z2} on the restricted
Burnside problem to conclude that there is a bound on the size of the
groups $\rho_j \circ \theta(A)$, so the chain in \eqref{chain}
eventually stabilizes.  This means that $\rho_j\circ
\theta$ is injective for some $j$, which means that $A$ is finite, as
desired.

  Now suppose instead that $S$ contains a morphism
  $f: \bP^n \lra \bP^n$ of degree greater than 1.  We will argue by
  contradiction, and show that if $Sx$ is infinite, then there must be
  some $s \in S$ such that $x \notin \Per(s)$.  Since $S$ is generated
  by finitely many elements we may assume that $K$ is finitely
  generated. We now argue as in \cite[Section 3]{BHPT}.  If $K$ is a
  number field, let $h$ denote a Weil height on $\bP^n$; if $K$ is
  not, let $h$ denote a Moriwaki height on $\bP^n$ (see \cite{Mori1,
    Mori2}).  For $s \in S$ of degree greater than 1, we let $h_s$ be
  the corresponding canonical height given by
  $h_s(z) = \lim_{n \to \infty} \frac{h(s^n(z))}{(\deg s)^n}$; note
  that this converges uniformly for all $z$ (see \cite{CS93}), so in
  particular $|h_s - h|$ is bounded.  Recall that $h_s(z) = 0$ if and
  only if $z$ is preperiodic under $s$ (see \cite[Corollary
  3.5]{BHPT}).  Now, if $Sx$ is not finite, then there is some
  $g \in S$ such that $h_f(gx) \not= 0$ by the Northcott property for
  canonical heights (see \cite{Northcott} and \cite[Section
  3.2.2]{BHPT}). Let $S'$ be the semigroup
  generated by $f$ and $fg$.  Then there is a $C$ such that
  $|h_s(z) - h_f(z)| < C$ for all $z \in \bP^n$. Since
  $h_f(f^ng x) = (\deg f)^n h_f(gx)$ and $h_f(gx) \not= 0$, there is
  an $n$ such that $h_f(f^n g x) > C$; this means that
  $h_{f^n g}(f^ng x) \not= 0$.  Thus, $f^n gx \notin \Prep(f^ng)$, so
  $x \notin \Prep(f^ng)$, and our proof is complete.
\end{proof}




\section{Nilpotent and solvable} \label{nil-sec}

 We will begin by proving Theorem \ref{nil-same}.  We start with two
 lemmas.  The first is standard.  See \cite[5.2.22]{Robinson} for a
 proof.  
 
 \begin{lem}\label{center}
   If $G$ is finitely generated, nilpotent, and infinite then its center $Z(G)$
   contains an element of infinite order.
   \end{lem}

   \begin{lem}\label{commute-same}
If $g, h: X \lra X$ are automorphisms such that $gh = hg$, then
$g(\Per_n(h)) \subseteq \Per_n(h)$ and $\Per_n^*(h) \subseteq \Per(g)$.  
\end{lem}
\begin{proof}
  If $h^n(x) = x$, then $h^n(g(x)) = g (h^n(x)) = g(x)$; if
  $h^m(x) \not= x$, then $h^m(g(x)) = g (h^m(x)) \not= x$ since $g$ is
  one-to-one.  Hence, $g(\Per_n(h)) \subseteq \Per_n(h)$, and
  $g(\Per^*_n(h)) \subseteq \Per^*_n(h)$. Thus, every element of
  $\Per^*_n(h)$ has a finite forward orbit under $g$ and is therefore
  in $\Per(g)$.
\end{proof}

We are now ready to prove Theorem \ref{nil-same}.  We need one more
definition.

\begin{defn}
Let $G$ be a group of automorphisms of a variety $X$. If $Y$ is a
subset of a variety $X$, we define the stabilizer $G_Y$ of $Y$ as
\[ G_Y := \{ g \in G \; | \; g(Y) = Y \}. \]
We have a restriction map $\Res_Y: G_Y \lra \Aut(Y)$ given by
\[ \Res_Y(g) = g |_Y \text{ for $g \in G_Y$}. \]
We let $H_Y$ be the image of this map, that is
\[ H_Y := \Res_Y(G_Y). \]

\end{defn}

\begin{proof}[Proof of Theorem \ref{nil-same}]
 
  We will use Noetherian induction, which states that if a proposition
  is true for a point and is true for a variety $Y$ whenever it is
  true for all proper subvarieties of $Y$, then it is true for all
  varieties (see \cite[Ex. II.3.16]{Har}, for example).  

  If $X$ is a point, we are done since $G$ must be finite.  For the
  inductive step, we take a nontorsion $h \in Z(G)$, the center of
  $g$; some such $h$ exists by Lemma \ref{center} since $G$ must now
  be infinite as otherwise $g_1$ and $g_2$ would be torsion, and the
  theorem would be trivial.  By Lemma \ref{commute-same}, we have
  $\Per^{*}(g_1) \subseteq \Per(h)$.  Thus, we will be done if we can
  show that for any $n$ and any element
  $x \in \Per_n(h) \cap \Per^{*}(g_1)$, we must have
  $x \in \Per(g_2)$.  Fix an $n$.  Then $\Per_n(h)$ is a proper closed
  subset of $X$, which we denote as $Y$ and we have
  $g_1(Y) = g_2(Y) = Y$ by Lemma \ref{commute-same}.  Thus, each point
  corresponding to a component of dimension 0 in $\Per_n(h)$ is
  periodic for both $g_1$ and $g_2$.  Now, let $x \in \Per_n(h)$ lie
  in a component $Y$ of $\Per_n(h)$ of positive dimension.  Let
  $g_i' = \Res_Y(g_m)$.  The stabilizer group $H_Y$ is finitely
  generated and nilpotent since it is the image of a subgroup of a
  finitely generated nilpotent group.  Now, $Y$ is a proper closed
  subset of $X$, so we may
  apply the inductive hypothesis to conclude that
  $\Per^{*}(g_1') \subseteq \Per(g_2')$. Now, if
  $x \in Y \cap \Per^{*}(g_1)$, then $x \in \Per^*(g_1')$ so we see
  that every element of $x \in Y \cap \Per^{*}(g_1)$ is in
  $\Per(g_2)$, as desired.
\end{proof}

In the case of group of automorphisms that are virtually nilpotent
(have a nilpotent subgroup of finite index) we must strengthen our
hypothesis, as the following example shows.  

\begin{example}\label{nil-ex}
Let $G$ be the group of affine linear maps of the form $x \mapsto
\pm x +m $ for $m \in \bZ$.  The subgroup $H$ consisting of all maps
$x \mapsto x +m$ is abelian of index 2 in $G$, so $G$ is virtually
nilpotent.  If $f(x) = -x$, then 0 is in $\Per_1^*(f)$, but 0 is not
periodic under any elements of $H$ other than the identity.  Thus the
condition that $\alpha \in \Per^{**}(g)$ for some $g$ and not just
$\alpha \in \Per^*(g)$ is necessary in Theorem \ref{nil-same}.  
\end{example}

We are able to prove the following, however.

\begin{cor}\label{ns2}
  Let $G$ be a finitely generated group of automorphisms of a variety
  $X$.  If $G$ is virtually nilpotent, then for any
  $g_1, g_2 \in G$, we have $\Per^{*}(g_1) \subseteq \Per(g_2)$.  In
  particular, if $\Per_n(g_1)$  and $\Per_n(g_2)$ are finite for all
  $n$, then $\Per(g_1) = \Per(g_2)$.  
 \end{cor}
 \begin{proof}
   We observe that while $\Per^*(g^n)$ may
   be smaller than $\Per^*(g)$ (as in Example \ref{nil-ex}), we always
   have both $\Per^{**}(g^n) = \Per^{**}(g)$ and $\Per(g^n) = \Per(g)$.
   Now, let $A$ be a nilpotent subgroup finite index in $G$ and let
   $g_1, g_2 \in G$; then there is some $n$ such that
   $g_1^n, g_2^n \in A$.  Since
   $\Per^{*}(g_1^n) \subseteq \Per(g_2^n)$ by Theorem \ref{nil-same},
   we then have
 \[ \Per^{**}(g_1) = \Per^{**}(g_1^n) \subseteq \Per^{*}(g_1^n) \subseteq
 \Per(g_2^n) = \Per(g_2). \]

 \end{proof}

 Now we are ready to prove Theorem \ref{auto-proj}.  
 \begin{proof}[Proof of Theorem \ref{auto-proj}]
   By \cite[Theorem 1.1]{Tits-Var-Z} in characteristic 0 and
   \cite[Theorem 1.1]{Hu} in characteristic $p$ (see also
   \cite{Tits-Var}), the group $G$ is either virtually solvable or it
   contains a free group on two generators.  Thus, we may assume that
   $G$ is virtually solvable.  By \cite[Theorems 4.2 and
   4.7]{Rosenblatt}, a virtually solvable group that is not virtually
   nilpotent contains a free semigroup on two generators.  If
   $\Per^{**}(g_1)$ is not contained in $\Per(g_2)$ for some
   $g_1, g_2 \in G$, then $G$ is not virtually nilpotent by Corollary
   \ref{ns2}, so $G$ must contain a free semigroup on two generators.
 \end{proof}

 We will now begin our proof of Theorem \ref{solvable}. We begin with a lemma.

 \begin{lem}\label{commutator}
   Let $A$ be a group of automorphisms of a quasiprojective variety
   $X$.  Let $A'$ denote $[A,A]$, let $x \in X$, and let $g,h \in A$
   Suppose that $A'x$ is finite and that $x \in \Per(h)$.  Then $gx
   \in \Per(h)$.  
 \end{lem}
 \begin{proof}
 Choose $\ell$ such that $h^\ell(x) = x$.  Since $g^{-1} h^{\ell m} g
 h^{-\ell m}  \in A'$ for all $m$ and $A'x$ is finite, we see that
 there are some $m_1 \not= m_2$ such that
 \[ g^{-1} h^{\ell m_1} g h^{-\ell m_1} (x) =  g^{-1} h^{\ell m_2} g
   h^{-\ell m_2} (x).\]
 Since $x$ is fixed by $h^\ell$ and we can cancel the $g^{-1}$ from
 each side, this means that $h^{\ell m_1}(g(x)) = h^{\ell
   m_2}(g(x))$, so $h^{\ell (m_1 - m_2)}(g(x)) = g(x)$, so $g(x) \in
 \Per(h)$.  
\end{proof}

\begin{prop}\label{commutator2}
Let $A$ be a subgroup of a finitely generated group of automorphisms
of a quasi-projective variety $X$. Let $A' =[A,A]$ and let $h \in
A$. Suppose that $x$ is periodic under every element of $A'$ and that
$x \in \Per^*(h)$ for some $h \in A'$.  Then $x \in \Per(g)$ for
every $g \in A$.  
\end{prop} 
\begin{proof}
  There is a finite extension $K_v$ of some $\bQ_p$ with ring of
  integers $\fo_v$ and a model $\cX$ for $X$ over $\fo_v$ such that
  $x$ extends to an element of $\cX(\fo_v)$ and every element of
  $g \in A$ extends to an automorphism of $\cX$. Let $T$ be the set of
  elements in $\cX(\fo_v)$ that are periodic under every element of
  $A'$, and let $N$ be the subgroup of $A'$ that fixes every element
  of $T$.  Then $N$ is normal of finite index in $A'$ by Proposition
  \ref{burn1}, so $A'x$ is finite.  Now, let $g \in A$.  We have
  $g(T) \subseteq T$ by Lemma \ref{commutator}, so we have
  $g N g^{-1} (t) = t$ for any $t \in T$.  Thus, for any coset $bN$ of
  $N$, we have $g bN g^{-1}= gbg^{-1} g N g^{-1} = gbg^{-1} N$, so $g$
  permutes the cosets of $N$ in $A'$.  Thus, there is an $m$ such that
  $g^m b N g^{-m} = b N$ for every coset $bN$ of $N$ in $A'$.

  Let $x \in \Per_n(h) \cap T$.  We have $g^m h^n N = h^n N g^m$, so
  $g^m(x) = g^m h^n(x) = h^n(g^m(x))$, so $g^m(z) \in\ Per_n(h) \cap T$
  as well.  Let $Z_1 = \Per_n(h)$, which is Zariski closed, let $Z_2$
  be the Zariski closure of $T$, and let $Z = Z_1 \cap Z_2$. We have
  $g^m(Z) = Z$, so in particular $g^m$ sends zero-dimensional
  components of $Z$ to themselves. If $x \in \Per^*(h)$ then $x$ is in
  a zero-dimensional component of $Z$, so $x$ is thus periodic under
  $g^m$, which means that it is also periodic under $g$.
\end{proof}

\begin{proof}[Proof of Theorem \ref{solvable}]

  It is enough to treat the case where $G$ is solvable, since if a
  finite index subgroup $A$ of $G$ contains a finite index subgroup
  $H$ such that $h(x) =x$ for some $x$ for all $h \in H$, then $H$ has
  finite index in $G$ as well.  Thus, we will assume from now on that
  $G$ is solvable and not merely virtually solvable.

   We proceed by induction on the length of the derived series for
   $G$.  We will prove something slightly more precise than the
   statement of Theorem \ref{solvable}.  We will show that if $G$ is a
   subgroup of a finitely generated group and every element of
   $\Per^*(g)$ is nonempty for all $g$, then there is an $x \in X$
   such that $x \in \Per(g)$ for all $g \in G$ and $x \in \Per^*(h)$
   for some $n$ and $h$.  

   If $G$ is abelian, then for any $g, h \in G$, we have
   $g(\Per_n(h)) \subseteq \Per_n(h)$, so if $x \in \Per^{*}(h)$ for
   some $h$, then $x \in \Per(g)$ for all $g \in G$, since any
   $g \in G$ permute the zero-dimensional components of the Zariski closure of
   $\Per_n(h)$.  Note that while commutator groups may not be finitely
   generated, they are still subgroups of our original finitely
   generated group so Proposition \ref{commutator2} applies. Now, let
   $G' = [G,G]$.  By the inductive hypothesis, there is an
   $x \in X$ such that $x \in \Per(g)$ for all $g \in G'$ and
   $x \in \Per^*(h)$ for some $h \in G'$.
   By Proposition
   \ref{commutator2}, we have that $x \in \Per(g)$ for all $g \in G$
   as well, and our proof is complete.
 \end{proof}

 The following example shows that we need to have $\Per^*(g)$ and not
 just $\Per(g)$ nonempty for all $g \in G$ in order to conclude there
 is a point that is periodic under every element of a solvable group $G$.

 \begin{example}\label{solv-ex}
Consider the abelian group of $A = \Aut(\bA^2)$ of maps of the form

\[ f_{a,b} (x,y)= (x+ay+a+b \sqrt(2)y,y) \] where $a$ and $b$ are
integers.  Then for integers $a,b$ this map fixes the line consisting
of points of the form $(z,c)$ where $c$ is a solution to
$(a+b \sqrt{2})c + a=0$. Notice this $c$ exists and is unique unless
$a=b=0$ in which case we get the identity map. Notice when $a=1, b=0$
the only periodic points are $(x,-1)$ and when $a=0, b=1$ the only
periodic points are $(x,0)$ so there are no common periodic points.
Thus, if $G$ is any finitely generated subgroup of $A$ containing
$f_{1,0}$ and $f_{0,1}$, then there is no $x \in \bA^2$ that is in
$\Per(g)$ for all $g \in G$. This shows that the condition that
$\Per^*(g)$ and not just $\Per(g)$ be nonempty for all $g$ is
necessary in Theorem \ref{solvable}.
\end{example}





\section{More general endomorphisms}\label{north-sec}

Let $f: X \lra X$ be any finite endomorphism.  We let $\Prep_{m,n}(f)$
be set of points $f^m(x) = f^{m+n}(x)$.  Again, this is a closed
subset of $X$.  We let $\Prep^*_{m,n}(f)$ be the points that are the
zero dimensional components of $\Prep_{m,n}(f)$ and we let
$\Prep^*(f)$ be the union of all the $\Prep^*_{m,n}(f)$.

We also let $\Prep^{**}(f)$ denote the elements of $\Prep(f)$ that are
not in a positive-dimensional component of $\Prep_{m,n}(f)$ for any
$f$, in analogy with the definition of $\Per^{**}(g)$ for an
automorphism $g$.

We pose the following question.

\begin{quest}
Let $S$ be a finitely generated semigroup of finite morphisms from a
quasiprojective variety to itself.  Is it true that at least one of
the following must hold:
\begin{itemize}
\item[(i)] we have $\Per^{**}(f_1) \subseteq \Per(f_2)$ for any $f_1,
  g_2 \in G$ or
\item[(ii)] $G$ contains a subsemigroup isomorphic to the free semigroup on
  two generators?
\end{itemize}
\end{quest}

While we cannot obtain results along the lines of Theorems
\ref{auto-proj}, \ref{nil-same}, \ref{solvable}, and \ref{arith-burn}
for general finite morphisms, we are able to obtain two Northcott-type
results.  The first can be thought of as a Northcott theorem for
integral points.

\begin{thm}\label{north}
  Let $f: X \lra X$ be a finite morphism, where $X$ is a projective variety
  defined over a finitely generated field $K$ of characteristic 0.
  Let $D$ be any positive integer.  Then there are at most finitely
  many $\alpha \in \Prep^*(f)$ such that $[K(\alpha): K] \leq D$.
\end{thm}

\begin{thm}\label{contain}
  Let $f, g: X \lra X$ be finite morphisms of a quasi-projective
  variety over a field of characteristic 0.  If
  $g(\Prep^*(f)) \subseteq \Prep^*(f)$, then
  $\Prep^*(f) \subseteq \Prep(g)$.

\end{thm}

We will prove Theorems \ref{north} and \ref{contain} after introducing
a few propositions and lemmas. 

Fakhruddin \cite[Proposition 1]{Fak2} (see also
\cite[Theorem 1.2]{Keping2}) proved the following.

\begin{prop} \label{period}
  Let $X$ be a variety defined over a finite extension $K_v$ of
  $\bQ_p$ and let $\cX$ be a model for $X$ over the ring of integers
  $\fo_v$ of $K_v$.  Let $f: \cX \lra \cX$ be a morphism.  Then there
  is a bound on the periods of elements of $\cX(\fo_v)$.
\end{prop}

The below comes from an argument of Scanlon (see \cite[Section
3]{Scanlon}).

\begin{prop}\label{preimage}
Let $X$ be a variety defined over a finite extension $K_v$ of
$\bQ_p$ and let $\cX$ be a model for $X$ over the ring of integers
$\fo_v$ of $K_v$.  Let $f: \cX \lra \cX$ be a finite map.  Let $\alpha
\in \cX(\fo_v)$ be periodic for $f$.  Then there are finitely many
$\beta \in \cX(\fo_v)$ such that there is an $n$ for which $f^n(\beta)
= \alpha.$
 \end{prop}
 \begin{proof}
   Let $f^m(\alpha) = \alpha$. Without loss of generality, we replace
   $f$ with $f^m$ and suppose that $f(\alpha) = \alpha$.  Pick a
   $\gamma \not= \alpha$ in $\cX(\fo_v)$ such that
   $f(\gamma) = \alpha$ (if no such $\gamma$ exists, we are done).
   Since there are at most finitely many such $\gamma$, it suffices to
   show that there are finitely many $\beta \in \cX(\fo_v)$ such that
   $f^n(\beta) = \gamma$ for some $n$.  Let $r_s$ be the reduction map
   from $\cX(\fo_v)$ to $\cX(\fo_v/\fm^s)$ where $\fm$ is the maximal
   ideal of $\fo_v$.  Then there is an $s$ such that
   $r_s(\gamma) \not= r_s(\alpha)$.  Let $f_s$ denote the map $f$
   induces on $\cX(\fo_v/\fm^s)$.  Then $r_s(\gamma)$ is not periodic
   under $f_s$.  Let $M$ be the number of points in $\cX(\fo_v/\fm^s)$
   (which is finite since $\fo_v / \fm^s$ is finite).  Then there is
   no $\beta$ such that $f^n(\beta) = \gamma$ for $n > M$.  This
   completes our proof.
 \end{proof}

 \begin{lem}\label{dim}
   Let $f: X \lra X$ be finite and surjective.  If
   $\alpha \in \Prep^*_{m,n}(f)$, then
   $f^m(\alpha) \in \Prep^*_{0,n}(f)$.
 \end{lem}
 \begin{proof}
   For any component $Y$ of $X$, its image $f(Y)$ has the same
   dimension as $Y$ because $f$ is finite, by the Going-Up Theorem
   (see \cite[Theorem 5.16]{AM}, for example).  Recall that finite
   maps are closed and affine (see \cite[Chapter II, Ex. 3.4 and
   3.5]{Har}). Since $X$ has finitely many components, it follows that
   for any component $Y$ of $X$, its image $f(Y)$ is a component of
   $X$.  Then for any open affine $U = \Spec A$ in $X$, there is a ring
   $B$ such that $f^{-1}(U)$ is isomorphic to $\Spec B$
   and the map $f^*$ from $A$ to $B$ is an inclusion with
   $\dim A = \dim B$ (again, by the Going-Up Theorem).  Hence, for any
   irreducible closed subvariety of $V$ of $X$, and any component $W$
   of $f^{-1}(V)$, we have $\dim V = \dim W$.  In particular, every
   0-dimensional component of $\Prep^*_{m,n}(f)$ is the inverse image
   of a 0-dimensional component of $\Prep^*_{0,n}(f)$.
   \end{proof}

   Now, we are ready to prove Theorems \ref{north} and \ref{contain}.

  \begin{proof}[Proof of Theorem \ref{north}]
    We embed $K$ into $\bZ_p$ as in Theorem \ref{arith-burn}.  Note
    that $\bQ_p$ has finitely many extensions of degree $D$ or less,
    so we let $K_v$ be their compositum.  By Lemma \ref{dim}, every
    point in $\Prep^*_{m,n}(f)$ is in $f^{-m}(\alpha)$ for some
    $\alpha \in \Prep^*_{0,n}(f)$.  By Proposition \ref{period}, the
    set 
    $\cup_{n=1}^\infty \Prep^*_{0,n}(f)$ is finite.  And by
    Proposition \ref{preimage}, for any $\alpha \in \Prep^*_{0,n}(f)$,
    the set of $\beta$ such that there is an $m$ for which
    $f^m(\beta) = \alpha$ is finite.  Thus, there are at most finitely
    many $\alpha \in \Prep^*(f)$ such that $[K(\alpha) : K] \leq D$,
    as desired.
  \end{proof}

  \begin{proof}[Proof of Theorem \ref{contain}]
    Let $\alpha \in \Prep^*(f)$.  Again embed into $\bZ_p$ in a way
    that $\alpha$ extends to a point $\ta$ in $\bZ_p$ and so that $f$
    and $g$ extend to maps on $\cX(\bZ_p)$. To prove that
    $\alpha \in \Prep(g)$, it will suffice to show that there are only
    finitely many points in $X(K) \cap \Prep^*(f)$ that extend to
    points in $\cX(\bZ_p)$, since $\alpha \in X(K) \cap \Prep^*(f)$
    and $g$ maps $X(K) \cap \Prep^*(f)$ to itself. For any
    $\beta \in X(K)$ that extends to a point in $\cX(\bZ_p)$ we let
    $\tb$ denote this extension (note that not all points in $X(K)$
    necessarily extend to points in $\cX(\bZ_p)$ since $X$ may not be
    projective).  By Lemma \ref{dim}, for any
    $\gamma \in \Prep_{m,n}^*(f)$ we have $f^m(\gamma) = \beta$ for
    some $\beta \in \Prep_{0,n}^*(f)$.  By Proposition \ref{period},
    the set of periods of points in $\cX(\bZ_p)$ is bounded.  For any
    periodic $\tb \in \cX(\bZ_p)$, there are only finitely many
    $\tg \in \cX(\bZ_p)$ such that $f^m(\tg) = \tb$ for some $m$, by
    Proposition \ref{preimage}.  Thus, we see that there are at most
    finitely many pairs $(m,n)$ such that there are points in
    $X(K) \cap \Prep_{m,n}^*(f)$ that extend to points in
    $\cX(\bZ_p)$.  Since each $\Prep_{m,n}^*(f)$ is finite, this
    finishes our proof.
  \end{proof}

  \newcommand{\etalchar}[1]{$^{#1}$}
\providecommand{\bysame}{\leavevmode\hbox to3em{\hrulefill}\thinspace}
\providecommand{\MR}{\relax\ifhmode\unskip\space\fi MR }
\providecommand{\MRhref}[2]{%
  \href{http://www.ams.org/mathscinet-getitem?mr=#1}{#2}
}
\providecommand{\href}[2]{#2}

\end{document}